\documentclass[amssymb, 11pt]{amsart}
\usepackage{latexsym}
\usepackage{young}

\newlength{\standardunitlength}
\newcommand{\bea}{\begin{eqnarray}}
\newcommand{\ena}{\end{eqnarray}}
\newcommand{\beas}{\begin{eqnarray*}}
\newcommand{\enas}{\end{eqnarray*}}
\newcommand{\qmq}[1]{\quad \mbox{#1} \quad}

\newcommand{\bbox}{\hfill $\Box$}
\newcommand{\nn}{\nonumber}

\newtheorem{prop}{Proposition}[section]

\newtheorem{lemma}[prop]{Lemma}

\newtheorem{theorem}[prop]{Theorem}

\newcommand{\ee}{\mathbb{E}}

\begin{document}

\begin{center} \title [Zero biasing and growth processes] {\bf Zero biasing and growth processes}
\end{center}

\author{Jason Fulman}
\address{University of Southern California \\  Los Angeles, CA 90089-2532}
\email{fulman@usc.edu}

\author{Larry Goldstein}
\address{University of Southern California \\
Los Angeles, CA 90089-2532}
\email{larry@math.usc.edu}

\keywords{Stein's method, zero biasing, Plancherel measure, Jack measure, random transposition, P\'{o}lya urn, growth process}

\date{May 13, 2011}

\thanks{Fulman was partially supported by National Science Foundation grant DMS 0802082 and National Security Agency
grant H98230-08-1-0133.}

\thanks{Goldstein was partially supported by National Security Agency
grant H98230-11-1-0162.}

\begin{abstract} The tools of zero biasing are adapted to yield a general result suitable for analyzing the behavior of certain growth processes. The main theorem is applied to prove central limit theorems, with explicit error terms in the $L^1$ metric, for certain statistics of the Jack measure on partitions and for the number of balls drawn in a P\'{o}lya-Eggenberger urn process.
\end{abstract}

\maketitle

\section{Introduction}

Zero biasing for the normal approximation of a random variable $W$ using Stein's method was introduced in Goldstein and Reinert \cite{GR}. One instance in which the zero bias method may be applied is for $W$ for which a Stein pair
$W,W'$ may be constructed, that is, for $W$ that may be coupled to a variable $W'$ such that $W,W'$ is exchangeable and
satisfies $E(W'|W)=(1-a)W$ for some $a \in (0,1]$. After giving a brief review of these methods in
Section \ref{review}, in Section \ref{general} we provide a general result allowing one to apply zero biasing when the statistic $W$ of interest is formed by certain growth processes and can be coupled in a Stein pair.

Section \ref{jack} studies a certain statistic $W_\alpha$ under the Jack$_{\alpha}$ measure on partitions.
We defer precise definitions to Section \ref{jack}, but for now mention that is of interest to
study statistical properties of Jack$_{\alpha}$ measure. The case $\alpha=1$ corresponds to the
actively studied Plancherel measure of the symmetric group. The surveys
\cite{AD},\cite{De}, \cite{O2} and the seminal papers \cite{BOO},\cite{J},\cite{O1} indicate how the Plancherel measure of
the symmetric group is a discrete analog of random matrix theory, and describe its importance in representation theory and geometry. Okounkov \cite{O2} notes that the study of Jack$_{\alpha}$ measure is an important open problem, about which relatively little is known. It is a discrete analog of Dyson's $\beta$ ensembles from random matrix theory \cite{BO1}.

The particular statistic $W_\alpha$ under Jack measure which we study is of interest for several reasons. When $\alpha=1$
it reduces to the character ratio of transpositions under Plancherel measure, or equivalently to the
spectrum of the random transposition walk. Also by Corollary 1 of \cite{DH}, there is a natural random
walk on perfect matchings of the complete graph on $n$ vertices, whose eigenvalues are precisely $\frac{W(\lambda)}{\sqrt{n(n-1)}}$, occurring
with multiplicity proportional to the Jack$_2$ measure of $\lambda$.
The proofs to date of central limit theorem for $W_\alpha$ range from combinatorial
ones using the method of moments in \cite{K1}, \cite{H}, \cite{Sn}, and the use Stein's method, which produces
an error term (but with no explicit constant) in the Kolomogorov metric \cite{F1}, \cite{F2}, \cite{SS}. Our contribution is to prove a central limit theorem in the $L^1$ metric, with a small explicit constant.

Section \ref{polya} applies the main result of Section \ref{general} to study a growth process arising from the P\'{o}lya-Eggenberger urn model. More precisely,
imagine an urn ${\mathcal U}_{A,B}$ containing $A$ white balls and $B$ black balls. At each time step one ball is drawn, and returned to the urn along
with $m$ balls of the same color. This is one of the simplest urn models, discussed in detail in the textbooks \cite{JK} and \cite{M}. We obtain a central limit theorem with explicit error term for the number of white balls drawn after $n$ steps. While \cite{JK} and \cite{M} contain many useful results and pointers to the literature, including some central limit theorems in more general settings, to the best of our knowledge the literature does not contain results that provide such error terms for this problem.

\section{Stein's method and zero biasing} \label{review}

Stein's lemma \cite{S1} states that a random variable $Z$ has the mean zero normal distribution ${\mathcal N}(0,\sigma^2)$
if and only if
\bea \label{Stein.ch}
\sigma^2 Ef'(Z)=E[Zf(Z)]
\ena
for all absolutely continuous functions $f$ for which these expectations exist. Motivated by this
characterization, for a mean zero, variance $\sigma^2$ random variable $W$ and a given function $h$ on which to test the difference between $Eh(W)$ and $Nh=Eh(Z)$, Stein \cite{S1} considered the differential equation
\bea \label{Stein.eq}
\sigma^2 f'(w)-wf(w) = h(w)-Nh.
\ena
For the unique bounded solution $h$ of (\ref{Stein.eq}), one can evaluate the required difference by
substituting $W$ for $w$ and taking expectation, to yield
\beas
Eh(W)-Nh=E[\sigma^2 f'(W)- Wf(W)].
\enas
Though it may not be immediately clear why the right hand side may be simpler to evaluate
than the left, a variety of techniques have been developed to handle various situations. For instance,
the exchangeable pair technique, from \cite{S2} handles the expectation of the right hand side when
the given random variable $W$ can be coupled to $W'$ so that $(W,W')$ is an $a$-Stein pair, that is, an
exchangeable pair that satisfies
\bea \label{WpW.Stein.pair}
E(W'|W)=(1-a)W \quad \mbox{for some $a \in (0,1)$.}
\ena
Other techniques for handling the Stein equation are discussed in detail in \cite{C1} and in the references therein, but of particular relevance here is the zero bias coupling, which we now review.

Though the mean zero normal is the unique distribution satisfying (\ref{Stein.ch}), one can ask whether
a given variable satisfies a like identity of it own. Indeed, it is shown in \cite{GR} that for every mean zero,
variance $\sigma^2$ random variable $X$, there exists a distribution for a random variable $X^*$, termed
the $X$-zero biased distribution, such that
\bea \label{zero.bia.char}
\sigma^2 Ef'(X^*)=E[Xf(X)]
\ena
for all absolutely continuous functions $f$ for which these expectations exist. The mapping of ${\mathcal L}(X)$, the
distribution of $X$, to ${\mathcal L}(X^*)$, is known as the zero bias transformation. In particular, Stein's lemma
(\ref{Stein.ch}) can be rephrased as the statement that the mean zero normal ${\mathcal N}(0,\sigma^2)$ is the unique fixed point of the zero bias transformation characterized by (\ref{zero.bia.char}).

Heuristically, then, if the transformation has a fixed point at the mean zero normal, then an approximate fixed point
should be approximately normal. This heuristic has been made precise for a variety of examples in \cite{GR}, \cite{G1}, \cite{G2}, \cite{G3} and \cite{G4} (see also \cite{C1}) in order to yield bounds in both the Kolmogorov and $L^1$ metric.
For the latter, the following result from \cite{G4} is often useful; we use $|| \cdot ||_1$ to denote the $L^1$ metric.
\begin{theorem}\label{zb2}
If the mean zero, variance 1 random variable $W$ can be coupled to $W^*$ having the $W$-zero bias distribution, then
\beas
||{\mathcal L}(W)-{\mathcal L}(Z)||_1 \le 2E|W^*-W|
\enas
where $Z$ is a standard normal variable.
\end{theorem}

Hence, to obtain $L^1$ bounds, the question reduces to finding a way to couple $W$ and $W^*$. Lemma \ref{GR} below of \cite{GR}, noting here that the result holds also for $a=1$, shows how the construction of a variable $W^*$ with the $W$-zero bias distribution can be achieved with
the help of the distribution $dF(w,w')$ of a Stein pair. First, it can easily be shown from (\ref{WpW.Stein.pair}) that if $W,W'$ is an $a$-Stein pair possessing second moments then
\bea \label{Stein.pair.a}
\ee W=0 \qmq{and} \ee(W'-W)^2=2a\mbox{Var}(W),
\ena
so in particular,
\bea \label{F.dagger.dist}
dF^\dagger(w,w')=\frac{(w'-w)^2}{2a}dF(w,w')
\ena
is a bivariate distribution.

\begin{lemma}
\label{GR}
If $W^\dagger, W^\ddagger$ have distribution
(\ref{F.dagger.dist}) where $F(w,w')$ is the joint distribution of an $a$-Stein pair, and $U$ is a uniformly distributed variable, independent of $W^\dagger,W^\ddagger$, then
\beas %\label{W*.linear.interpolation}
W^*=UW^\dagger + (1-U) W^\ddagger
\enas
has the $W$-zero bias distribution.
\end{lemma}
In particular, if $W$ and $W^\dagger,W^\ddagger$ can be constructed on a common space,
then $W$ and $W^*$ can be also.

We remark that a number of results are available when $(W,W')$ is only an approximate Stein pair, that is, an exchangeable pair that satisfies the linearity condition (\ref{WpW.Stein.pair}) with a remainder, see for instance \cite{RR}, and \cite{C1}. Correspondingly, here we expect the conclusions of Theorems \ref{zb2} and \ref{T.cond.V} to hold for approximate Stein pairs by including in the bounds the additional terms that arise from such remainders.

In what follows we study processes for which the random variable $W$ of interest can be
written as the sum $V+T$, where $V$ is a function of a variable $\tau$ determined by the process run to a penultimate state, and $T$ a function
of running the process for one additional step. In our examples, given $\tau$, a Stein pair $(W,W')=(V+T,V+T')$ can be constructed by running two copies of the last step of chain, forming $T$ and $T'$ conditionally independent given $\tau$.

In such cases a pair of random variables with distribution (\ref{F.dagger.dist}) can be similarly constructed by forming $(W^\dagger,W^\ddagger)=(V^\Box+T_{\tau^\Box}^\dagger,V^\Box+T_{\tau^\Box}^\ddagger)$ for $V^\Box$ and $T_{\tau^\Box}^\dagger,T_{\tau^\Box}^\ddagger$ sampled by biasing the distributions of $V$ and $T,T'$ in a certain way. Our first application of Theorem \ref{T.cond.V}, to Jack measure, is particularly simple since the biasing factor to form the $V^\Box$ distribution from that of $V$ is unity, and we may therefore take $V=V^\Box$. For our second example, the P\'{o}lya-Eggenberger urn, we will see that biasing draws from the urn ${\mathcal U}_{A,B}$ in our process results in the urn ${\mathcal U}_{A+m,B+m}$.

\section{General Result} \label{general}

The purpose of this section is to prove the following theorem.

\begin{theorem} \label{T.cond.V}
Consider a bivariate distribution ${\mathcal L}(\tau,T)$ on a random object $\tau$ and random variable $T$, and a $\tau$ measurable random variable $V=V_\tau$ such that sampling $\tau$, and then, given $\tau$, sampling  $T$ and $T'$ independently from the conditional distribution ${\mathcal L}(T|\tau)$, the random variables
\bea \label{WWp.sum}
W=V+T \qmq{and} W'=V+T'
\ena
have variance one and are an $a$-Stein pair.
Denoting
\bea \label{ETWgV.noV}
\ee(T|\tau)=\mu_\tau \qmq{and} \ee((T-\mu_\tau)^2|\tau) = \sigma_\tau^2,
\ena
and the distribution of $\tau$ by $dF(\tau)$, the measure $F^\Box(\tau)$ specified by
\bea \label{dist:vtau}
dF^\Box(\tau)=\frac{\sigma_\tau^2}{a}dF(\tau)
\ena
is a probability measure, and for any coupling of $\tau$ to $\tau^\Box$ with distribution (\ref{dist:vtau}), we have
\bea \label{L1.T.cond.V}
\lefteqn{||{\mathcal L}(W)-{\mathcal L}(Z)||_1}\nonumber \\
&&\le 2 \ee |(V_{\tau^\Box}-V) +(\mu_{\tau^\Box}-\mu_\tau)|+ 2 \ee |T-\mu_\tau|+\frac{\ee |T-\mu_\tau|^3}{\mbox{Var}(T-\mu_\tau)}.
\ena
When $\mu_\tau$ equals zero and $\sigma_\tau^2$ is constant almost surely, then
\bea \label{L2.T.cond.V}
||{\mathcal L}(W)-{\mathcal L}(Z)||_1 \le 2 \ee|T|+\frac{\ee|T^3|}{\mbox{Var}(T)}.
\ena
\end{theorem}

\begin{proof} First consider the case where $\mu_\tau=0$ a.s..
Since conditional on $\tau$ the pair $T$ and $T'$ are independent,
we have $\ee[T'T|\tau]=\ee[T'|\tau]\ee[T|\tau]=0$, and therefore, from (\ref{WWp.sum}) and (\ref{ETWgV.noV}),
\bea \label{expT-Tp2}
\ee((W'-W)^2|\tau)=\ee((T'-T)^2|\tau)=2\sigma_\tau^2.
\ena
Taking expectation and applying (\ref{Stein.pair.a}), we have that
\bea \label{Esig2}
\ee \sigma_\tau^2=a,
\ena verifying that $dF^\Box(\tau)$ is a probability measure.

By construction, the joint distribution of $(T,T',\tau)$ is, with some abuse of notation, given by
\beas
dF(t,t',\tau)=dF(t'|\tau)dF(t|\tau)dF(\tau),
\enas
and therefore the pair $(W,W')$ has distribution
\bea \label{dFwwp}
dF(w,w')=\int_{\tau,t,t': v+t=w,v+t'=w'}dF(t'|\tau)dF(t|\tau)dF(\tau),
\ena
where $v=V_\tau$.
By Lemma \ref{GR}, with $U$ an independent uniform random variable on $[0,1]$,
$$
W^*=UW^\dagger+(1-U)W^\ddagger
$$
has the $W$-zero bias distribution when $(W^\dagger,W^\ddagger)$ has distribution given by
\beas
dF^\dagger(w,w')=\frac{(w'-w)^2}{2a}dF(w,w').
\enas

For any fixed $\tau$ let $F(t|\tau)$ denote
the conditional distribution of $T$ given $\tau$. By (\ref{expT-Tp2}), for every $\tau$ the measure
\bea \label{dFtau.tt.daggers}
dF_\tau^\dagger(t,t')=\frac{(t'-t)^2}{2\sigma_\tau^2}dF(t'|\tau)dF(t|\tau),
\ena
is a bivariate probability distribution.

Now using (\ref{dFwwp}), (\ref{Esig2}) and (\ref{dFtau.tt.daggers})
\bea
\lefteqn{dF^\dagger(w,w')}\nn \\
&=& \frac{(w'-w)^2}{2a}\int_{\tau,t,t':v+t=w,v+t'=w'}dF(t'|\tau)dF(t|\tau)dF(\tau)\nn \\
&=& \int_{\tau,t,t':v+t=w,v+t'=w'}\frac{(w-w')^2}{2a} dF(t'|\tau)dF(t|\tau)dF(\tau)\nn \\
&=& \int_{\tau,t,t':v+t=w,v+t'=w'}\frac{\sigma_\tau^2}{a}\frac{(t'-t)^2}{2\sigma_\tau^2} dF(t'|\tau)dF(t|\tau)dF(\tau)\nn \\
&=& \int_\tau  \left( \int_{t,t':v+t=w,v+t'=w'} \frac{(t'-t)^2}{2\sigma_\tau^2} dF(t'|\tau)dF(t|\tau) \right) \frac{\sigma_\tau^2}{a} dF(\tau) \label{mixing.measure}\\
&=& \int_\tau  \left( \int_{t,t':v+t=w,v+t'=w'} dF_\tau^\dagger(t',t) \right) dF^\Box(\tau).\nn
\ena

The factorization in the integral indicates that given $\tau^\Box$ with distribution $dF^\Box(\tau)$,
the pair $(W^\dagger, W^\ddagger)$ can be generated by sampling $T_{\tau^\Box}^\dagger,T_{\tau^\Box}^\ddagger$ from $dF_{\tau^\Box}^\dagger(t',t)$, and then
setting
\beas
W^\dagger=V_{\tau^\Box}+T_{\tau^\Box}^\dagger \qmq{and} W^\ddagger=V_{\tau^\Box}+T_{\tau^\Box}^\ddagger,
\enas
where $V_{\tau^\Box}$ is the value of $V$ on $\tau^\Box$. In particular, letting
\bea \label{TBox}
T^{\tau^\Box}=UT_{\tau^\Box}^\dagger+(1-U)T_{\tau^\Box}^\ddagger,
\ena
we have that
\beas
W^*=U(V_{\tau^\Box}+T_{\tau^\Box}^\dagger)+(1-U)(V_{\tau^\Box}+T_{\tau^\Box}^\ddagger)=
V_{\tau^\Box}+T^{\tau^\Box}
\enas
has the $W$-zero biased distribution.

For a fixed $\tau$, let $T_\tau$ and $T_\tau'$ denote independent copies of
a random variable with distribution $dF(t|\tau)$. Clearly $T_\tau$ and $T_\tau'$ are exchangeable, and as $\mu_\tau=0$, we have
$\ee(T)=\ee \left( \ee(T|\tau) \right)= \ee \mu_\tau = 0$ and therefore $\ee(T'|T)=\ee(T')= 0$.
Hence $(T,T')$ is a $1$-Stein pair. In view of (\ref{dFtau.tt.daggers}), Lemma \ref{GR} yields that when $T_\tau^\dagger,T_\tau^\ddagger$ have
distribution $F_\tau^\dagger(t,t')$ and $U$ is an independent uniform random variable,
\bea \label{Tnu}
T_\tau^*=UT_\tau^\dagger + (1-U)T_\tau^\ddagger
\ena
has the $T_\tau$-zero biased distribution.

As $\ee(T)=0$, by (\ref{Esig2}) we obtain
\beas
a=\ee \sigma_\tau^2 = \ee \left( \ee(T^2|\tau)\right) = \ee (T^2) = \mbox{Var}(T).
\enas
Comparing (\ref{TBox}) and (\ref{Tnu}), we see that the distribution
${\mathcal L}(T^{\tau^\Box})$ is the mixture of the distributions ${\mathcal L}(T_\tau^*)$ with mixing measure
 $\sigma_\tau^2/\mbox{Var}(T)$, by (\ref{mixing.measure}). Therefore, by Theorem 2.1 of \cite{G3}, $T^{\tau^\Box}$ has the $T$-zero bias distribution.
Applying the zero
bias identity (\ref{zero.bia.char}) with $f(x)=(1/2)x^2 \mbox{sign}(x)$, we have
\beas
\ee|T^{\tau^\Box}|= \frac{\ee|T^3|}{2\mbox{Var}(T)}.
\enas

Now, with $\tau$ and $\tau^\Box$ the given coupling, letting $V=V_\tau$ and $T$ be sampled from ${\mathcal L}(T|\tau)$,
 setting $(W,W^*)=(V+T,V^\Box+T^{\tau^\Box})$ yields a coupling
of $W$ and $W^*$ on the same space, satisfying
\beas
\lefteqn{\ee|W^*-W|} \\
&=&    \ee|V_{\tau^\Box}-V+T^{\tau^\Box}-T|  \\
&\le&  \ee|V_{\tau^\Box}-V|+\ee|T|+\ee|T^{\tau^\Box}|  \\
&=&    \ee|V_{\tau^\Box}-V|+ \ee|T|+\frac{\ee|T^3|}{2\mbox{Var}(T)}.
\enas

Theorem \ref{zb2} now yields
\bea \label{L1.bd.mu0}
||{\mathcal L}(W)-{\mathcal L}(Z)||_1  \le  2\ee|V_{\tau^\Box}-V|+2 \ee|T|+\frac{\ee|T^3|}{\mbox{Var}(T)}.
\ena
When $\sigma_\tau^2$ is constant we have that $dF^\Box(\tau)=dF(\tau)$, and hence may let $\tau^\Box=\tau$; taking $V_{\tau^\Box}=V$ in (\ref{L1.bd.mu0}) now yields (\ref{L2.T.cond.V}).

To obtain the result for general $\mu_\tau$, we reduce to the case $\mu_\tau=0$ by writing
\beas
(W,W')=(V+T,V+T')=((V+\mu_\tau)+(T-\mu_\tau),(V+\mu_\tau)+(T'-\mu_\tau)).
\enas
Replacing $V$ and $T$ in (\ref{L1.bd.mu0}) by $V+\mu_\tau$ and $T-\mu_\tau$, respectively, yields (\ref{L1.T.cond.V}).
\end{proof}

\section{The Jack measure} \label{jack}

In this section we apply Theorem \ref{T.cond.V}
to study a property of the Jack$_{\alpha}$ measure on the set of
partitions of size $n$. For $\alpha>0$ the Jack$_{\alpha}$
measure chooses a partition $\lambda$ of size $n$ with probability
\bea \label{def:Jackalpha}
\mbox{Jack}_\alpha(\lambda)=
\frac{\alpha^n n!}{\prod_{x \in \lambda} (\alpha a(x) + l(x) +1)
(\alpha a(x) + l(x) + \alpha)},
\ena
where in the product over all boxes $x$
in the partition $\lambda$, $a(x)$ denotes the number of boxes in the same
row of $x$ and to the right of $x$ (the ``arm'' of $x$), and $l(x)$
denotes the number of boxes in the same column of $x$ and below $x$
(the ``leg'' of $x$). For example one calculates that the partition
\[ \lambda=
\begin{array}{c c c} \framebox{}& \framebox{}& \framebox{} \\
\framebox{}& \framebox{}& \end{array} \] of 5 has Jack$_{\alpha}$
measure
\[ \mbox{Jack}_\alpha(\lambda)= \frac{60 \alpha^2}{(2 \alpha+2)(3 \alpha+1) (\alpha+2)(2
\alpha+1)(\alpha+1)}.\]

With $\lambda$ having the $\mbox{Jack}_\alpha$ distribution, we apply the theory of Section \ref{general} to prove
an explicit $L_1$ normal approximation bound for the statistic
\beas %\label{def:WalphaJack}
W_{\alpha}(\lambda) = \frac{\sum_{x \in
\lambda} c_{\alpha}(x)}{\sqrt{\alpha {n \choose 2}}}
\enas
where $c_{\alpha}(x)$ denotes the
``$\alpha$-content'' of $x$, defined as
\[ c_{\alpha}(x) = \alpha(\mbox{column number of $x - 1$}) - (\mbox{row number of $x - 1$}).\]
In the diagram below representing a partition of 7, each box is filled with its $\alpha$-content: \[
\begin{array}{c c c c} \framebox{0}& \framebox{$\alpha$}& \framebox{2 $\alpha$}&
\framebox{3 $\alpha$} \\ \framebox{$-1$}& \framebox{$\alpha-1$}&& \\ \framebox{$-2$} &&&
\end{array}. \]

In the Kolmogorov metric, the paper \cite{F1} proved an $O(n^{-1/4})$ error term for the normal
approximation of $W_{\alpha}$; this rate was sharpened in \cite{F4} using martingales to $O(n^{(-1/2)+\epsilon})$ for any $\epsilon>0$ and in \cite{F3} to $O(n^{-1/2})$ using Bolthausen's inductive
approach to Stein's method, but without an explicit constant. The text \cite{HO} proves a central limit theorem, with no error term, for $W_{\alpha}$ using quantum probability. Here we give an explicit $L_1$ bound to the normal with small constants.

To obtain our bound we construct an exchangeable pair using Kerov's growth process for generating a random partition distributed according to Jack$_\alpha$ measure. Given a box $x$ in the diagram of $\lambda$, again letting $a(x)$ and $l(x)$ denote the arm and leg of $x$ respectively, set
\beas c_{\lambda}(\alpha) =
	\prod_{x \in \lambda} (\alpha a(x) + l(x) +1), \quad
	c_{\lambda}'(\alpha) = \prod_{x \in \lambda} (\alpha a(x) +
	l(x) + \alpha)
\enas
and, for $\tau$ a partition obtained from $\lambda$ by removing a single corner box,
 \[ \psi_{\lambda/\tau}'(\alpha) = \prod_{x \in
C_{\lambda/\tau}-R_{\lambda/\tau}} \frac{(\alpha a_{\lambda}(x) +
l_{\lambda}(x)+1)}{(\alpha a_{\lambda}(x) + l_{\lambda}(x)+\alpha)}
\frac{(\alpha a_{\tau}(x) + l_{\tau}(x)+\alpha)}{(\alpha a_{\tau}(x) +
l_{\tau}(x)+1)}
\]
where $C_{\lambda/\tau}$ is the union of columns of $\lambda$ that
intersect $\lambda-\tau$ and $R_{\lambda/\tau}$ is the union of rows of $\lambda$ that intersect $\lambda-\tau$.

The state of Kerov's growth process at times $n=1,2,\ldots$ is a partition of size $n$, starting at time one with the unique partition of 1. If at stage $n-1$ the state of the process is the partition $\tau$, a transition to the partition $\lambda$ occurs with probability
\[ \frac{c_{\tau}(\alpha)}{c_{\lambda}(\alpha)} \psi_{\lambda/\tau}'(\alpha).\] As shown in \cite{K2}, \cite{F4}, if $\tau$ is chosen from the Jack$_{\alpha}$ measure on partitions of size $n-1$, then transitioning according to this rule results in a partition $\lambda$ of $n$ distributed according to Jack$_{\alpha}$ measure.

We now present an $L^1$ bound for the normal approximation of $W_\alpha$.

\begin{theorem} \label{jackbound} Let
\bea
\label{def:Walpha}
W_{\alpha}(\lambda)=\frac{\sum_{x \in
\lambda} c_{\alpha}(x)}{\sqrt{\alpha {n \choose 2}}}
\ena
and let $W_\alpha$ be the value of $W_\alpha(\lambda)$ when $\lambda$
has the Jack$_{\alpha}$ measure distribution for some $\alpha>0$. Then for $Z$
a standard normal random variable,
\bea \label{L1.Plancherel}
||{\mathcal L}(W_{\alpha})-{\mathcal L}(Z)||_1 \le \sqrt{\frac{2}{n}} \left( 2 + \sqrt{2+\frac{\max(\alpha,1/\alpha)}{n-1}} \right).
\ena
\end{theorem}

\begin{proof} First we show (\ref{L1.Plancherel}) holds for all $\alpha \ge 1$.
Constructing $\tau$ from the Jack measure on partitions of size $n-1$ and then taking one
step in Kerov's growth process yields $\lambda$ with the Jack measure on partitions of size $n$,
and we may write
\beas
W_{\alpha}=V+T
\enas
where
\beas
V= \frac{\sum_{x \in
\tau} c_{\alpha}(x)}{\sqrt{\alpha {n \choose 2}}} \qmq{and} T=\frac{c_{\alpha}(\lambda/\tau)}{\sqrt{\alpha {n \choose 2}}},
\enas
and $c_{\alpha}(\lambda/\tau)$ denotes the $\alpha$-content of the box added to $\tau$ to form $\lambda$.

It is shown in \cite{F1} that constructing $\lambda'$ by taking another step in Kerov's growth
process from $\tau$, independently of $\lambda/\tau$ given $\tau$, and then forming $W_{\alpha}'$ from $\lambda'$ as $W$
is formed from $\lambda$, results in exchangeable variables $W_{\alpha},W_{\alpha}'$ that satisfy (\ref{WpW.Stein.pair}) with $a=2/n$.
Hence, (\ref{WWp.sum}) of Theorem \ref{T.cond.V} is satisfied. Corollary 5.3 of \cite{F1} gives that $Var(W)=1$.

From Section 3 of \cite{F3}, one recalls the following three facts:
\begin{enumerate}
\item \label{Jack.1} $E[T|\tau]=0$ for all $\tau$.
\item \label{Jack.2} $E[T^2|\tau]=\frac{2}{n}$ for all $\tau$.
\item \label{Jack.3} $E[T^4]= \frac{\alpha^2 {n \choose 2} + \alpha(\alpha-1)^2(n-1) + 3 \alpha^2 {n-1 \choose 2}}{\alpha^2 {n \choose 2}^2}$
\end{enumerate}

As $V$ is measurable with respect to the $\sigma$-algebra generated by $\tau$, condition (\ref{ETWgV.noV})
is satisfied. From properties (\ref{Jack.1}) and (\ref{Jack.2}) above we have, respectively, that $\mu_\tau=0$ and $\sigma_\tau^2$ is a constant, almost surely. Hence the bound (\ref{L2.T.cond.V}) of Theorem \ref{T.cond.V} holds.

Applying the Cauchy-Schwarz inequality gives that $E|T| \leq \sqrt{ET^2} = \sqrt{2/n}$, accounting
for the first term in the bound. From property (\ref{Jack.3}), now applying $\alpha \geq 1$, we have
\[ E[T^4] \leq \left[ \frac{{n \choose 2} + 3 {n-1 \choose 2}}{{n \choose 2}^2} \right] + \frac{\alpha(n-1)}{{n \choose 2}^2} \leq \frac{8}{n^2} + \frac{4 \alpha}{n^2(n-1)}. \] The Cauchy-Schwarz
inequality gives that $E|T^3| \leq \sqrt{E[T^2] E[T^4]}$, and properties (\ref{Jack.1}) and (\ref{Jack.2}) give
$\mbox{Var}(T)=2/n$, yielding the final term in the bound (\ref{L1.Plancherel}). Thus the result is shown
when $\alpha \ge 1$.

To obtain a bound for all $\alpha>0$ note first that when taking the transpose $\lambda^t$ of a partition $\lambda$ the roles of the arms $a(x)$ and legs $l(x)$ become interchanged; hence, letting $\lambda_\alpha$ be a partition with the $\mbox{Jack}_\alpha$ distribution,
from (\ref{def:Jackalpha}), for all $\alpha>0$ we have
\beas
{\mathcal L}(\lambda_\alpha)={\mathcal L}(\lambda_{1/\alpha}^t).
\enas
Next, as $W_\alpha(\lambda)=-W_{1/\alpha}(\lambda^t)$ for all $\lambda$, and ${\mathcal L}(Z)={\mathcal L}(-Z)$,
\beas
\lefteqn{||{\mathcal L}(W_\alpha(\lambda_\alpha))-{\mathcal L}(Z)||_1}\\
&=&||{\mathcal L}(-W_{1/\alpha}(\lambda_\alpha^t))-{\mathcal L}(Z)||_1\\
&=&||{\mathcal L}(-W_{1/\alpha}(\lambda_{1/\alpha}))-{\mathcal L}(-Z)||_1\\
&=&||{\mathcal L}(W_{1/\alpha}(\lambda_{1/\alpha}))-{\mathcal L}(Z)||_1.
\enas
Hence, as the bound (\ref{L1.Plancherel}) holds for all $\alpha \ge 1$, it holds for all $\alpha>0$.
\end{proof}

\section{P\'{o}lya-Eggenberger urn model} \label{polya}

For $m,n,A,B>0$ fixed integers, we define a probability distribution on the
set $\{0,1,\cdots,n\}$ by
\bea \label{def:Mnk}
M_{n,A,B}(k) = {n \choose k} \frac{(A/m)_k (B/m)_{n-k}}{(A/m+B/m)_n}.
\ena
Unless clarity demands it, we will simply write $M_n(k)$ for $M_{n,A,B}(k)$.
Here $x_r=x(x+1) \cdots (x+r-1)$, the rising factorial,
where we set $x_0$=1.

It is well known \cite{K3}, \cite{M}, \cite{JK} that the distribution $M_n(k)$ can be achieved
in the following way. Imagine an urn ${\mathcal U}_{A,B}$ that initially has $A$ white and $B$ black balls. At each time step,
one ball is drawn uniformly from the urn and then returned back along with $m$ balls of the same color. If $S_n$ is the number of white balls drawn in the first $n$ draws, then
\beas
P(S_n=k)=M_n(k) \quad \mbox{for $k=0,1, \ldots, n$.}
\enas
We note that when $S_n=k$ the urn ${\mathcal U}_{A,B}$ contains $A+km$ white balls.

In this section we prove the following $L^1$ normal approximation to the distribution of $S_n$, properly standardized.

\begin{theorem}
\label{thm:mainAB}
For $n \in \mathbb{N}$ let $S_n$ be the number of white balls added to ${\mathcal U}_{A,B}$ after $n$ time steps, and set
\bea \label{def:WnAB}
W_n= \sqrt{\frac{(A+B+m)n}{AB(A+B+nm)}} \left[ A - \frac{(A+B)S_n}{n} \right].
\ena
Then $W_n$ has mean zero and variance 1, and for $Z$ a standard normal random variable, for $n \ge (A+B+m)/2m$
\beas
\lefteqn{||{\mathcal L}(W_{n+1})-{\mathcal L}(Z)||_1}\\
&\le&  \left(\frac{4mn}{A+B+m}+  \frac{A^2+6AB+B^2}{AB} \right)\sqrt{\frac{(A+B+m)^3}{AB(A+B+nm+m)(n+1)}}
\enas
while for $n<(A+B+m)/2m$,
\beas
\lefteqn{||{\mathcal L}(W_{n+1})-{\mathcal L}(Z)||_1}\\
&\le& \left( \frac{A^2+8AB+B^2}{AB} \right)\sqrt{\frac{(A+B+m)^3}{AB(A+B+nm+m)(n+1)}}.
\enas
\end{theorem}

From Theorem 3.2 of \cite{M}, we know with $A,B,m$ fixed and $n \rightarrow \infty$,
\beas
S_n/n \rightarrow_d {\mathcal B}(A/m,B/m), \quad \mbox{}
\enas
that is, the fraction of white balls drawn converges to the Beta distribution with parameters
$A/m,B/m$. In particular, the limiting value of the bound as $n \rightarrow \infty$, giving
an $L^1$ bound between the standardized Beta distribution and the normal, is
$4\sqrt{m(A+B+m)/AB}$; for, say $A=B$, the bound specializes to $4\sqrt{m(2A+m)}/A$, which tends
to zero at rate $1/\sqrt{A}$ if $m$ is fixed and $A$ grows.

For what follows it is useful to relate the distribution $M_n(k)$ to up and
down chains. On the set $\Gamma_n= \{(n,k): 0 \leq k \leq n \}$,
placing directed edges from $(n-1,k)$ to $(n,k)$ and to $(n,k+1)$ results in what is known
as known as Pascal's lattice \cite{K3}. It is convenient to define $d((n,k))={n \choose k}$,
the number of paths from $(0,0)$ to $(n,k)$.
More generally, one defines $d((n,k) / (m,j))$ to be the number of
paths from $(m,j)$ to $(n,k)$; this is ${n-m \choose k-j}$.

We define an ``up'' chain that transitions from $(n,k)$ to $(n+1,k)$ with probability
$(B+nm-km)/(A+B+nm)$ and to $(n+1,k+1)$ with probability $(A+km)/(A+B+nm)$. We also define
a ``down'' chain that transitions from $(n,k)$ to $(n-1,k-1)$ with probability $k/n$
and to $(n-1,k)$ with probability $1-(k/n)$. One easily checks that if $(n-1,k)$ is distributed
according to $M_{n-1}$, then applying the up chain gives an element of $\Gamma_n$ distributed according to
$M_n$. Similarly, if $(n,k)$ is distributed according to $M_n$, one checks that applying
the down chain gives an element of $\Gamma_{n-1}$ distributed according to $M_{n-1}$.

We denote the up chain from $\Gamma_{n-1}$ to $\Gamma_n$ by $U_{n-1}$ and the down chain
from $\Gamma_n$ to $\Gamma_{n-1}$ by $D_n$. A straightforward computation yields that
\begin{equation} \label{commut} D_{n+1}U_n = c_n U_{n-1}{D_n} + (1-c_n) I_n \end{equation}
with $c_n=\frac{n(A+B+nm-m)}{(n+1)(A+B+nm)}$, so that the tools of \cite{F5} are in force.

The following lemma shows how to use the up and down chains to construct a Stein pair, that is, a pair of
exchangeable random variables satisfying (\ref{WpW.Stein.pair}).
\begin{lemma} \label{expair:AB}
Let $W_n$ be given by (\ref{def:WnAB}) with
$S_n$ the number of white balls added to ${\mathcal U}_{A,B}$ after $n$ time steps. Now construct $S_n'$ by transitioning
down using $D_n$ and then up using $U_{n-1}$, and let $W_n'$ be given by (\ref{def:WnAB})
with $S_n$ replaced by $S_n'$. Then $W_n,W_n'$ is an $a_n$-Stein pair with
\beas %\label{def:an}
a_n = \frac{A+B}{n(A+B+nm-m)}.
\enas
\end{lemma}

\begin{proof} By Theorem 4.3 of \cite{F5} and equation \eqref{commut}, a left eigenvector with eigenvalue $1-a_n$ is obtained by applying the operator $U^{n-1}$ to $(1,0)-(1,1)$. From the general theory of down-up chains (see \cite{F5}), one has that
\begin{eqnarray*} U^{n-1}(1,0) & = & \sum_{k=0}^n \frac{M_n(k) d((n,k)/(1,0))}{M_1(0) d(n,k)} \cdot (n,k) \\
& = & \sum_{k=0}^n M_n(k) \frac{{n-1 \choose k} (A+B)}{{n \choose k} B} \cdot (n,k). \end{eqnarray*} Similarly,
\begin{eqnarray*} U^{n-1}(1,1) & = & \sum_{k=0}^n \frac{M_n(k) d((n,k)/(1,1))}{M_1(1) d(n,k)} \cdot (n,k) \\
& = & \sum_{k=0}^n M_n(k) \frac{{n-1 \choose k-1} (A+B)}{{n \choose k} A} \cdot (n,k). \end{eqnarray*}

Since $U_{n-1}D_n$ is a reversible Markov chain with stationary distribution $M_n$, its
right eigenvectors are obtained from its left eigenvectors by dividing by $M_n$. Thus \[ \frac{{n-1 \choose k} (A+B)}{{n \choose k} B}
 - \frac{{n-1 \choose k-1} (A+B)}{{n \choose k} A} = \frac{A+B}{AB} \left[ A - \frac{k(A+B)}{n} \right] \]
is a right eigenvector of $U_{n-1}D_n$ with eigenvalue $\left(1 - \frac{A+B}{n(A+B+nm-m)} \right)$. Since $W_n(k)$ is
a scalar multiple of $\frac{A+B}{AB} \left[ A - \frac{k(A+B)}{n} \right]$, the result follows. \end{proof}

The next goal is to compute the mean and variance of $W_n$ given by (\ref{def:WnAB})
with $S_n$ the number of white balls drawn in the first $n$ draws. Clearly for all $n \ge 1$ one may write
\beas
S_n = {\bf 1}_0 + \cdots + {\bf 1}_{n-1}
\enas
where ${\bf 1}_j=1$ if a white ball is drawn at time $j$, and ${\bf 1}_j=0$ otherwise.
The next lemma computes the mean and covariance of the indicators ${\bf 1}_j$.

\begin{lemma} \label{term}
For $j=0,\ldots,n-1,$ let ${\bf 1}_j$ denote the indicator that a white ball is drawn from ${\mathcal U}_{A,B}$ at time $j$. Then
\begin{enumerate}
\item $E[{\bf 1}_j] = \frac{A}{A+B}$ for all $j \in \{0,\ldots,n-1\}$.
\item $E[{\bf 1}_h {\bf 1}_j]= \frac{A(A+m)}{(A+B)(A+B+m)}$ for all $0 \le h < j \le n-1$
\item $E[S_n]=\frac{nA}{A+B}$.
\end{enumerate}
\end{lemma}

\begin{proof} It is classical and elementary that the indicators ${\bf 1}_j, j=0,\ldots,n-1$ are an exchangeable sequence (see
\cite{JK} or \cite{M} for a proof). Thus $E[{\bf 1}_j]$ is the probability that the first ball drawn is white, and
$E[{\bf 1}_h {\bf 1}_j]$ is the probability that the first two balls drawn are white. These observations, and linearity of expectation, yields the lemma.
\end{proof}

With the help of Lemma \ref{term}, we now compute the mean and variance of $W_n$.
\begin{lemma}
\label{mean.var.W.AB}
If $W_n$ is given by (\ref{def:WnAB}) where $S_n$ is the number of white balls added to ${\mathcal U}_{A,B}$ after
$n$ time steps, then
\beas
E[S_n^2]=\frac{nA}{A+B} + 2 {n \choose 2} \frac{A(A+m)}{(A+B)(A+B+m)},
\enas
\beas
EW_n=0 \qmq{and} \mbox{Var}(W_n) = 1.
\enas
\end{lemma}

\begin{proof}  Since $W_n,W_n'$ is a Stein pair we have that $EW_n=0$ by (\ref{Stein.pair.a}). Now, using the fact that ${\bf 1}_i^2={\bf 1}_i$, and both parts of Lemma \ref{term}, we obtain
\beas
E[S_n^2] & = & E[({\bf 1}_0+\cdots+{\bf 1}_{n-1})^2] \\
& = & E \left[ \sum_{i=0}^{n-1} {\bf 1}_i + 2 \sum_{0 \leq h<j \leq n-1} {\bf 1}_h{\bf 1}_j \right]\\
& = & \frac{nA}{A+B} + 2 {n \choose 2} \frac{A(A+m)}{(A+B)(A+B+m)},
\enas
yielding the first claim.

Applying the expression for $E[S_n]$ given by Lemma \ref{term}, it follows that
\begin{eqnarray*}
\mbox{Var}(S_n) & = & \left[ \frac{nA}{A+B} + 2 {n \choose 2} \frac{A(A+m)}{(A+B)(A+B+m)} - \frac{n^2A^2}{(A+B)^2} \right] \\
& = & \frac{AB(A+B+nm)n}{(A+B+m)(A+B)^2}. \end{eqnarray*} Hence, from the definition (\ref{def:WnAB}) of $W_n$ we conclude
that \[ \mbox{Var}(W_n) = \frac{(A+B+m)(A+B)^2}{AB(A+B+nm)n} \mbox{Var}(S_n) = 1.\]
\end{proof}

We will apply Theorem \ref{T.cond.V} by writing $W_{n+1}=V+T$ where
\bea \label{def:V}
V = \sqrt{\frac{(A+B+m)(n+1)}{AB(A+B+nm+m)}} \cdot \left[A - \frac{(A+B)S_n}{n+1} \right]
\ena
and
\bea \label{def:urnT}
T = - \sqrt{\frac{(A+B+m)}{AB(A+B+nm+m)(n+1)}} \cdot (A+B) \cdot {\bf 1}_n,
\ena
and letting $\tau=S_n$. We note that the condition in Theorem \ref{T.cond.V} that $V$ be $\tau$ measurable is here clearly satisfied.
The following lemma gives the properties of $T$ needed for computing an $L^1$ bound using Theorem \ref{T.cond.V}.

\begin{lemma} \label{t1}
Let $T$ be given by (\ref{def:urnT}) and $\tau=S_n$.
\begin{enumerate}

\item \label{condmeanAB} The conditional mean $\mu_\tau=E(T|\tau)$ is given by
\beas
\mu_\tau=- \sqrt{\frac{(A+B+m)}{AB(A+B+nm+m)(n+1)}} \cdot \frac{(A+B)  (A+mS_n)}{A+B+mn}.
\enas

\item \label{cond.var.T.urn} The conditional variance $\sigma_\tau^2=E((T-\mu_\tau)^2|\tau)$ is given by
\beas
\sigma_\tau^2 = \frac{(A+B+m)(A+B)^2}{AB(A+B+nm+m)(n+1)}  \frac{(A+mS_n)(nm-mS_n+B)}{(A+B+mn)^2}.
\enas

\item \label{urn.std.2rd.t} The variance $\mbox{Var}(T-\mu_\tau)$ satisfies
\beas
\mbox{Var}(T-\mu_\tau)=  \frac{(A+B)}{(n+1)\left( A+B+nm\right)}.
\enas

\item \label{EcondT-mutau} The absolute deviation of $T$ about $\mu_\tau$ satisfies
\beas
E|T-\mu_\tau| \leq \sqrt{\frac{(A+B+m)(A+B)^2}{AB(A+B+nm+m)(n+1)}}.
\enas

%\item \label{EcondT-mutau3}
%\beas
%E|T-\mu_\tau|^3 \leq \left[ \frac{A+B+m}{AB(A+B+nm+m)(n+1)} \right]^{3/2} \cdot (A+B)^3
%\enas

%\item \label{VarT} $\mbox{Var}(T) = \frac{A+B+m}{(A+B+nm+m)(n+1)}$.

%\item \label{urn.std.3rd.t} The third order deviation of $T$ about $\mu_\tau$, standardized by $\mbox{Var}(T)$,
%satisfies
%\beas
%\frac{E|T-\mu_{\tau}|^3}{\mbox{Var}(T)} \leq  \sqrt{\frac{(A+B+m)(A+B)^2}{AB(A+B+nm+m)(n+1)}} \cdot %\frac{(A+B)^2}{AB}.
%\enas

\item \label{urn.std.3rd.t.-mu} The third order deviation of $T$ about $\mu_\tau$, standardized by $\mbox{Var}(T-\mu_\tau)$,
satisfies
\beas
\frac{E|T-\mu_{\tau}|^3}{\mbox{Var}(T-\mu_\tau)} \leq  \sqrt{\frac{(A+B+m)^3}{AB(A+B+nm+m)(n+1)}} \cdot \frac{(A+B)^2}{AB}.
\enas

\end{enumerate}
\end{lemma}

\begin{proof} Parts \ref{condmeanAB} and \ref{cond.var.T.urn} follow immediately from (\ref{def:urnT}) and that
\bea \label{P1jgivenk}
P({\bf 1}_j=1|S_j=k) = \frac{A+mk}{A+B+mj}
\ena
for all $j=0,\ldots,n-1,k=0,\ldots,j$.

For part (\ref{urn.std.2rd.t}), first note that as $E(T-\mu_\tau|\tau)=0$ we have $\mbox{Var}(T-\mu_\tau)=E(T-\mu_\tau)^2$.  Now again using (\ref{P1jgivenk}), we have that $E\left( (T-\mu_\tau)^2 |S_n \right)$ equals
\beas
\lefteqn{\left( \frac{(A+B+m)(A+B)^2}{AB(A+B+nm+m)(n+1)} \right)  E\left[ \left( {\bf 1}_n - \frac{A+mS_n}{A+B+nm}\right)^2| S_n \right]}\\
&=& \left( \frac{(A+B+m)(A+B)^2}{AB(A+B+nm+m)(n+1)} \right) \left( \frac{(A+mS_n)(B+m(n-S_n))}{\left( A+B+nm\right)^2} \right).
\enas
Expanding the product $(A+mS_n)(B+m(n-S_n))$, taking expectation using the expressions for $E[S_n]$ and $E[S_n^2]$ provided by Lemmas \ref{term} and \ref{mean.var.W.AB}, respectively, the claim follows after some simplification.

For part \ref{EcondT-mutau}, one has that
\begin{eqnarray*}
& & E|T-\mu_\tau| \\
& = & E[E|T-\mu_\tau||S_n] \\
& = & \sqrt{\frac{A+B+m}{AB(A+B+nm+m)(n+1)}} \cdot (A+B) \\
& & \cdot E \left[ \left| {\bf 1_n}- \frac{A+mS_n}{A+B+nm} \right| | S_n \right] \\
& \leq & \sqrt{\frac{A+B+m}{AB(A+B+nm+m)(n+1)}} \cdot (A+B). \end{eqnarray*} The second equality used (\ref{P1jgivenk}),
and the inequality that
\bea \label{1n-E1n}
E \left[ \left| {\bf 1_n}- \frac{A+mS_n}{A+B+nm} \right|^p | S_n \right] \leq 1 \quad \mbox{for all $p \ge 0$}
\ena
with $p=1$.

Now, for part \ref{urn.std.3rd.t.-mu}, similarly, applying (\ref{1n-E1n}) with $p=3$ we obtain
\begin{eqnarray*}
E|T-\mu_\tau|^3 & = & E[E|T-\mu_\tau|^3|S_n] \\
& = & \left[ \frac{A+B+m}{AB(A+B+nm+m)(n+1)} \right]^{3/2} (A+B)^3 \\
 & & \cdot E \left[ \left| {\bf 1_n}- \frac{A+mS_n}{A+B+nm} \right|^3 | S_n \right] \\
& \leq & \left[ \frac{A+B+m}{AB(A+B+nm+m)(n+1)} \right]^{3/2} (A+B)^3.
\end{eqnarray*}
Part \ref{urn.std.3rd.t.-mu} now follows from part \ref{urn.std.2rd.t} by division. 

%Next, from the definition of $T$
%\beas %\label{calc.var.T.urn}
%\mbox{Var}(T) & = & \frac{A+B+m}{AB(A+B+nm+m)(n+1)} \cdot (A+B)^2 \cdot \mbox{Var}({\bf 1}_n) \\
%& = & \frac{A+B+m}{(A+B+nm+m)(n+1)}, \nn
%\enas
%where the second equality
%uses (1) of Lemma \ref{term}. Part \ref{urn.std.3rd.t} follows by division.
\end{proof}

Specializing (\ref{dist:vtau}) to the case at hand, with $M_{n,A,B}(k)$ the distribution of $S_n$ given by (\ref{def:Mnk}), we now consider constructing a coupling of $S_n$ to a random variable $S_n^\Box$ with distribution
\bea \label{def:Mnbox}
M_{n,A,B}^\Box(k) = \frac{\sigma_k^2}{a_{n+1}}M_{n,A,B}(k)
\ena
where $a_{n+1}$ is given by Lemma \ref{expair:AB}. The next result shows that one can achieve a variable with distribution $S_n^\Box$ by adding
 $2m$ additional balls to the urn at time zero, $m$ white and $m$ black, that is, by using the urn ${\mathcal U}_{A+m,B+m}$.

\begin{lemma} \label{MnMnBox}
\beas
M_{n,A,B}^\Box = M_{n,A+m,B+m}.
\enas
\end{lemma}

\begin{proof}
From (\ref{cond.var.T.urn}) of Lemma \ref{t1} and Lemma \ref{expair:AB}, we have
\beas
\frac{\sigma_k^2}{a_{n+1}}= \frac{(A/m+B/m)(A/m+B/m+1)(A/m+k)(B/m+n-k)}{(A/m)(B/m)(A/m+B/m+n)(A/m+B/m+n+1)}.
\enas
Hence, for all $k \in \{0,\ldots,n\}$,
\beas
\lefteqn{M_{n,A,B}^\Box (k)}\\
&=& \frac{(A/m+B/m) (A/m+B/m+1) (A/m+k)(B/m+n-k)}{(A/m)(B/m)(A/m+B/m+n)(A/m+B/m+n+1)} \\
& & \cdot {n \choose k}\frac{(A/m)_k(B/m)_{n-k}}{(A/m+B/m)_n}\\
&=&{n \choose k}\frac{(A/m+1)_k(B/m+1)_{n-k}}{(A/m+B/m+2)_n}\\
&=&M_{n,A+m,B+m}(k).
\enas
\end{proof}
Lemma  \ref{MnMnBox} shows that for the process $S_n$ on the urn ${\mathcal U}_{A,B}$, the process $S_n^\Box$ is for the urn ${\mathcal U}_{A+m,B+m}$. As for both processes no additional balls have been added at time zero, we have that
\bea \label{S0SBox0}
S_0=0 \qmq{and} S_0^\Box=0.
\ena
As at times $n \ge 1$ both of these chains increase by $1$ when a white ball has been selected, if
$S_n=k$ and $S_n^\Box=j$, then $S_{n+1}=k+1$ and $S_{n+1}^\Box=j+1$ with respective probabilities
\bea \label{def:snksnkbox}
s_n(k)= \frac{A+km}{A+B+mn} \qmq{and} s_n^\Box(j)=\frac{A+m+jm}{A+B+2m+mn}.
\ena

We now couple $S_n$ and $S_n^\Box$ by coupling, at each stage, the two Bernoulli variables that indicate the drawing of a white ball in each urn. In particular, we couple these two Bernoulli variables so that the chance they are not equal is minimized.

\begin{theorem} \label{thm:ABcoupling}
Let $s_n(k)$ and $s_n^\Box(j)$ be given by (\ref{def:snksnkbox}) for $n,j,k \in \mathbb{N}$. Then
 the bivariate chain taking values in $\mathbb{N} \times \mathbb{N}$ characterized by
the initial condition $(S_0,S_0^\Box)=(0,0)$ and transitions
\beas
p_{n+1,n}(u,v|k,j)=P(S_{n+1}=u,S_{n+1}^\Box=v|S_n=k,S_n^\Box=j)
\enas
at times $n \ge 0$ according to
\beas
p_{n+1,n}(u,v|k,j)=\left\{
\begin{array}{cl}
\min(s_n(k),s_n^\Box(j))                  & (u,v)=(k+1,j+1)\\
(s_n^\Box(j)-s_n(k))^+                    & (u,v)=(k,j+1)\\
(s_n(k)-s_n^\Box(j))^+                    & (u,v)=(k+1,j)\\
1 - \max(s_n(k),s_n^\Box(j))         & (u,v)=(k,j)
\end{array}
\right.
\enas
is a coupling on a joint space of the urn models ${\mathcal U}_{A,B}$ and ${\mathcal U}_{A+m,B+m}$, respectively.

In addition, letting
\beas
N=\inf\{n \ge 1: S_n \not = S_n^\Box \}
\enas
we have
\bea \label{absSNSNBoxis1}
|S_N - S_N^\Box|=1
\ena
and
\beas
\mbox{if $S_N^\Box=S_N+1$ then $S_n^\Box \ge S_n$ for all $n \ge 0$,}
\enas
while, otherwise,
\beas
\mbox{if $S_N=S_N^\Box+1$ then $S_n \ge S_n^\Box$ for all $n \ge 0$}.
\enas

\end{theorem}

\begin{proof} That we must have $(S_0,S_0^\Box)=(0,0)$ is clear by (\ref{S0SBox0}).
As marginally for $S_n$ we have
\beas
P(S_{n+1}=k+1|S_n=k)= \min(s_n(k),s_n^\Box(j))  + (s_n(k)-s_n^\Box(j))^+   = s_n(k),
\enas
and similarly for $S_n^\Box$, both marginal transition functions
agree with those specified by (\ref{def:snksnkbox}), hence the joint chain is a coupling
of the two urn models in question.
Further, since $S_0=S_0^\Box$, and at most one white ball is drawn from either of the two urns at each time $n \ge 0$, (\ref{absSNSNBoxis1}) holds.

Taking the difference between the probabilities of drawing a white ball from either of the two urns yields
\bea \nn
\lefteqn{s^\Box_n(j)-s_n(k)}\\
&=&m \left(\frac{(A+mn) (j-k-1)+B (j-k+1) +2m(n -k)}{(A+B+mn) (A+B+2m+mn)}\right).\label{sdifference}
\ena
Suppose now that $S_N^\Box = S_N + 1$. We show by induction that $S_n^\Box \ge S_n+1$ for all $n \ge N$. Clearly
the claim is true for $n=N$. Assume that $S_n^\Box \ge S_n+1$ for some $n \ge N$, say $(S_n,S_n^\Box)=(k,j)$ with $j-k \ge 1$. Then, by (\ref{sdifference}) we see that
$s_n^\Box(j) \ge s_n(k)$, and
hence $(S_{n+1},S_{n+1}^\Box)$ equals $(k+1,j+1),(k,j+1)$ or $(k,j)$ with respective
probabilities $s_n(k),s_n^\Box(j)-s_n(k)$ and $1-s_n^\Box(j)$. In particular, $S_{n+1}^\Box \ge S_{n+1}+1$.

As the same argument applies in the case $S_n \ge S_n^\Box+1$, and since $S_n^\Box=S_n$ for all $0 \le n < N$
by the definition of $N$, the final claim of the lemma is shown.
\end{proof}

We now compute a bound on $E|S_n^\Box-S_n|$ for the coupling provided by  Theorem \ref{thm:ABcoupling}.

\begin{lemma}
The joint chain $(S_n,S_n^\Box)$ as specified in Theorem \ref{thm:ABcoupling} satisfies
\beas
E|S_n^\Box-S_n| \le \frac{2mn}{A+B+m}{\bf 1}\left(n \ge \frac{A+B+m}{2m}\right) +{\bf 1}\left(n < \frac{A+B+m}{2m}\right).
\enas
\end{lemma}

\begin{proof} By Theorem \ref{thm:ABcoupling}, with $N$ as defined there, we have
\beas
|S_n^\Box-S_n| = (S_n^\Box-S_n){\bf 1}_{\{n \ge N, S_N^\Box = S_N+1\}}+(S_n-S_n^\Box){\bf 1}_{\{n \ge N, S_N = S_N^\Box+1\}}.
\enas

For the first expectation,
\beas
\lefteqn{E \left[(S_n^\Box-S_n){\bf 1}_{\{n \ge N, S_N^\Box = S_N+1\}} \right]}\\
&=&\sum_{t=1}^{n-1} E \left[ (S_n^\Box-S_n){\bf 1}_{\{N=t, S_N^\Box = S_N+1\}} \right] +P(N=n, S_N^\Box = S_N+1)\\
&=&\sum_{t=1}^{n-1} \sum_{u \ge 0} E \left[ (S_n^\Box-S_n){\bf 1}_{\{N=t, S_N^\Box = S_N+1, S_N=u\}} \right]+P(N=n, S_N^\Box = S_N+1)\\
&=& \sum_{t=1}^{n-1} \sum_{u  \ge 0} E\left( S_n^\Box-S_n |N=t,S_N^\Box = S_N+1, S_N=u \right)\\
&& \cdot P(N=t,S_N^\Box = S_N+1, S_N=u)+P(N=n, S_N^\Box = S_N+1).
\enas

For $1 \le t \le n-1$, on the conditioning event, urn ${\mathcal U}_{A,B}$ has $A+mu$ white balls and $B+mt-mu$ black balls
at time $t$, and then has been run for time $n-t$. At each of these time steps, by Lemma \ref{term},
there is probability $(A+mu)/(A+B+mt)$ that a white ball will be selected from urn ${\mathcal U}_{A,B}$.

Similarly, for
$1 \le t \le n-1$, on the conditioning event, urn ${\mathcal U}_{A+m,B+m}$ has $A+m+(mu+m)=A+mu+2m$ white balls and
$B+m+mt-(mu+m)=B+mt-mu$ black balls
at time $t$, and then has been run for time $n-t$. At each of these time steps, by Lemma \ref{term},
the probability is $(A+mu+2m)/(A+B+mt+2m)$ that a white ball is selected from urn ${\mathcal U}_{A+m,B+m}$.

Hence, as it may be that all the balls chosen from ${\mathcal U}_{A,B}$ before time $N$ are black,
that is, we may have $S_N=u$ for $u=0$, we have
\beas
\lefteqn{E\left( S_n^\Box-S_n |N=t,S_N^\Box = S_N+1, S_N=u \right)}\\
&=& (n-t) \left( \frac{A+mu+2m}{A+B+mt+2m}-\frac{A+mu}{A+B+mt} \right)\\
&=& (n-t)\left(\frac{2m(B+mt-mu)}{(A+B+mt) (A+B+mt+2m)}\right)\\
&\le& 2m(n-t)\left(\frac{(B+mt)}{(A+B+mt) (A+B+mt+2m)}\right)\\
&\le& \frac{2m(n-t) (B+mt)}{(A+B+mt)^2}\\
&\le& \frac{2m(n-t)}{A+B+mt}.
\enas

Therefore
\beas
\lefteqn{E \left[ (S_n^\Box-S_n){\bf 1}_{\{n \ge N, S_N^\Box = S_N+1\}} \right]}\\
&\le& \sum_{t=1}^{n-1} \frac{2m(n-t)}{A+B+mt} \sum_{u  \ge 0} P(N=t,S_N^\Box = S_N+1, S_N=u)\\
&&+ P(N=n, S_N^\Box = S_N+1) \\
&=& \sum_{t=1}^{n-1} \frac{2m(n-t)}{A+B+mt} P(N=t,S_N^\Box = S_N+1) + P(N=n, S_N^\Box = S_N+1).
\enas

Reversing the roles of $S_n$ and $S_n^\Box$, though here noting that
it is necessary that $u \le t-1$
for the event $\{N=t,S_N = S_N^\Box+1, S_N^\Box=u\}$ to have positive probability, we similarly obtain
\beas
\lefteqn{E \left[(S_n-S_n^\Box){\bf 1}_{\{n \ge N, S_N = S_N^\Box+1\}} \right]}\\
&\le& \sum_{t=1}^{n-1} \frac{2m(n-t)}{A+B+mt} P(N=t,S_N = S_N^\Box+1) + P(N=n, S_N = S_N^\Box+1).
\enas
Now using that $(n-t)/(A+B+mt)$ is a decreasing
function of of $t \ge 0$, summing yields
\beas
\lefteqn{E|S_n^\Box-S_n|}\\
&\le& 2m\sum_{t=1}^{n-1} \frac{n-t}{A+B+mt} P(N=t) + P(N=n)\\
&\le& \frac{2mn}{A+B+m}P(N \le n-1) + P(N=n)\\
&\le& \frac{2mn}{A+B+m} {\bf 1}\left(n \ge \frac{A+B+m}{2m}\right) +{\bf 1}\left(n < \frac{A+B+m}{2m}\right),
\enas
as claimed, where in the final inequality we have used the fact that since $\alpha + \beta \le 1$
for $\alpha=P(N \le n-1)$ and  $\beta=P(N=n)$, then for any nonnegative numbers $a$ and $b$ we have
$\alpha a + \beta b \le \max(a,b)$.

\end{proof}

{\em Proof of Theorem \ref{thm:mainAB}}. That $EW_n=0$ and $\mbox{Var}(W_n)=1$ is
the content of Lemma \ref{mean.var.W.AB}.

We now compute the $L^1$ bound using Theorem \ref{T.cond.V}.
Applying (\ref{condmeanAB}) of Lemma \ref{t1} with $\tau^\Box$ and $\tau$ we obtain
\beas
\lefteqn{|\mu_{\tau^\Box}-\mu_\tau|}\\
&=&\sqrt{\frac{(A+B+m)(A+B)^2}{AB(A+B+nm+m)(n+1)}}  \cdot \frac{(A+m|S_n^\Box-S_n|)}{A+B+mn}\\
&\le& \sqrt{\frac{(A+B+m)(A+B)^2}{AB(A+B+nm+m)(n+1)}},
\enas
since both $S_n$ and $S_n^\Box$ take values between $0$ and $n$.

Applying the definition (\ref{def:V}) of $V$ on $\tau^\Box$ and $\tau$,
\beas
|V_{\tau^\Box} - V| = \sqrt{\frac{(A+B+m)(A+B)^2}{AB(A+B+nm+m)(n+1)}} \cdot |S_n^\Box-S_n|,
\enas
so that for $n \ge (A+B+m)/2m$ we have
\beas
E|V_{\tau^\Box} - V| = \frac{2mn}{A+B+m}\sqrt{\frac{(A+B+m)(A+B)^2}{AB(A+B+nm+m)(n+1)}},
\enas
while for $n < (A+B+m)/2m$,
\beas
E|V_{\tau^\Box} - V| = \sqrt{\frac{(A+B+m)(A+B)^2}{AB(A+B+nm+m)(n+1)}}.
\enas
The calculation is completed by using (\ref{EcondT-mutau}) and (\ref{urn.std.3rd.t.-mu}) of Lemma \ref{t1} for the final two terms, and then applying the inequality $(A+B+m)(A+B)^2 \le (A+B+m)^3$.
\bbox

 \end{document}